\documentclass[12pt]{amsart}
\usepackage{amssymb,amsmath,amsfonts,latexsym}
\usepackage{bm}
\usepackage[all,cmtip]{xy}
\usepackage{amscd}

\newtheorem{Theorem}{\sc Theorem}[section]
\newtheorem{Lemma}[Theorem]{\sc Lemma}

\theoremstyle{definition}

\theoremstyle{plain}

\theoremstyle{definition}

\numberwithin{equation}{section}

\begin{document}
\title[ An improved Copson inequality]
{ An improved Copson inequality}

\author[Das]{Bikram Das}
\address{Dr. APJ Abdul Kalam Technical University, Lucknow-226021 \&
Indian Institute of Carpet Technology, Chauri Road, Bhadohi- 221401, Uttar Pradesh, India}
\email{dasb23113@gmail.com}
\author[Manna]{Atanu Manna$^*$ }
\address{Indian Institute of Carpet Technology, Bhadohi-221401, Uttar Pradesh, India}
\email{atanu.manna@iict.ac.in, atanuiitkgp86@gmail.com ($^*$Corresponding author)}


\begin{abstract}
  In this paper, we prove that the discrete Copson inequality (E.T. Copson, \emph{Notes on a series of positive terms}, J. London Math. Soc., 2 (1927), 49-51) of one-dimension in general cases admits an improvement. In fact we study the improvement of the following Copson's inequality
\begin{align*}
&\displaystyle\sum_{n=1}^{\infty}\frac{Q_{n}^{\alpha}|A_n-A_{n-1}|^{2}}{q_{n}}\geq\frac{(\alpha-1)^2}{4}\displaystyle\sum_{n=1}^{\infty}
\frac{q_{n}}{Q_{n}^{2-\alpha}}|A_{n}|^{2},
\end{align*}where $\alpha\in[0,1)$, $A_{n}=q_{1}a_{1}+ q_{2}a_{2}+ \ldots +q_{n}a_{n}$, $Q_{n}=q_1+q_2+\ldots+q_{n}$ for $n\in \mathbb{N}$, $\{q_n\}$ is a positive real sequence and $\{a_n\}$ is a sequence of complex numbers. We show that if $\{q_n\}$ is decreasing then the above inequality has an improvement for $\alpha\in [1/3, 1)$. We also prove that for some increasing sequences $\{q_n\}$ the above inequality can also be improved. Indeed, we prove that for $q_{n}=n$ and $q_n=n^3$, $n\in \mathbb{N}$ the corresponding Copson inequalities admit an improvement for $\alpha\in[\frac{17}{50}, 1)$ and $\alpha\in[0, \frac{1}{2}]$, respectively. Further, we show that in case of $q_{n}=1$, $n\in \mathbb{N}$ the reduced Copson inequality (known as Hardy's inequality with power weights) has achieved an improvement for $\alpha\in[0,1)$.
\end{abstract}
\maketitle

\section{Introduction}{\label{secint}}
Suppose that $A=\{A_n\}\in C_{c}(\mathbb{N}_0)$, space of finitely supported functions on $\mathbb{N}_0$ with the assumption that $A_0=0$. Then the renowned discrete Hardy's inequality (\cite{GHYLITTLE}, Theorem 326) in one dimension states that
\begin{align}{\label{EFDHI}}
&\displaystyle\sum_{n=1}^{\infty}|A_n-A_{n-1}|^{2}\geq\frac{1}{4}\displaystyle\sum_{n=1}^{\infty}\frac{|A_n|^{2}}{n^{2}},
\end{align}
where the associated constant term is sharp. The inequality (\ref{EFDHI}) has a rich history and wider applications. Due to its enormous applications in multiple branches of mathematics such as differential equations, graph theory, and spectral theory, the inequality (\ref{EFDHI}) has been extended in various ways. One of the such valuable extensions is given by Hardy himself (\cite{GHYNOTE251}, \cite{GHYNOTE252}), which says that for a strictly positive sequence $\{q_n\}$ of real numbers if $\{\sqrt{q_n}a_{n}\}\in \ell_2$ holds then
\begin{align}{\label{GDHI}}
\displaystyle\sum_{n=1}^{\infty}q_nQ_{n}^{-2}|A_{n}|^{2}&<4\displaystyle\sum_{n=1}^{\infty}q_{n}|a_{n}|^{2},
\end{align} unless all $a_n$ is null, and where $A_{n}=q_{1}a_{1}+ q_{2}a_{2}+ \ldots +q_{n}a_{n}$, $Q_{n}=q_1+q_2+\ldots+q_{n}$ for $n\in \mathbb{N}$. Also the constant is best possible. Another important extensions is given by Hardy and Littlewood \cite{GHYLITTELE} as below:
\begin{align}{\label{HLI}}
\displaystyle\sum_{n=1}^{\infty}n^{-c}(\sum_{k=1}^{n}a_k)^2\leq K(c)\sum_{n=1}^{\infty}n^{-c}(na_n)^2,
\end{align}where $p>1$, $c\neq 1$ and $K=K(c)$ is a constant depending on $c$ such that RHS of inequality (\ref{HLI}) is convergent. Motivated by the inequalities (\ref{GDHI}), and (\ref{HLI}), Copson \cite{ECN2} (see also \cite{ECN}) established a sharp inequality as below:
\begin{align}{\label{alphaCOPI}}
&\displaystyle\sum_{n=1}^{\infty}\frac{Q_{n}^{\alpha}|A_n-A_{n-1}|^{2}}{q_{n}}\geq\frac{(\alpha-1)^2}{4}
\displaystyle\sum_{n=1}^{\infty}\frac{q_{n}}{Q_{n}^{2-\alpha}}|A_{n}|^{2},
\end{align} where $\alpha\in[0,1)$ with $\alpha=2-c$ for $1<c\leq 2$. The inequality (\ref{alphaCOPI}) is known as Copson inequality.\\
The recent surprising discovery of Keller, Pinchover and Pogorzelski \cite{MKR} (see also \cite{MKRGRAPH}) suggests that the inequality (\ref{EFDHI}) is not optimal in the following sense
\begin{align}{\label{IMHI2}}
\displaystyle\sum_{n=1}^{\infty}| A_n-A_{n-1}|^{2}&\geq\displaystyle\sum_{n=1}^{\infty} w_n^{KPP}|A_n|^{2}>\frac{1}{4}\displaystyle
\sum_{n=1}^{\infty} \frac{|A_n|^{2}}{n^{2}},
\end{align}where $w_n^{KPP}=2-\sqrt{\frac{n-1}{n}}-\sqrt{\frac{n+1}{n}}>\frac{1}{4n^2}$, $n\in \mathbb{N}$. This notable contribution has stimulated the interest of research in the direction of improvement of various discrete Hardy-type inequalities in one dimension by several mathematicians and researchers (see e.g. \cite{SGA}, \cite{HUANG21}, \cite{HUANGYE}). A short proof of (\ref{IMHI2}) with optimality was presented by Krej\v{c}i\v{r}\'{i}k and \v{S}tampach \cite{DKK}. Later on, Krej\v{c}i\v{r}\'{i}k, Laptev and \v{S}tampach (see \cite{DKALFS}, Theorem 10) have jointly established an extended version of improved discrete Hardy inequality (\ref{IMHI2}), and proved that
\begin{align}{\label{IMHIGN3}}
\displaystyle\sum_{n=1}^{\infty}| A_n-A_{n-1}|^{2}&\geq \displaystyle\sum_{n=1}^{\infty} w_n(\mu)|A_n|^{2},
\end{align}
where $w_n(\mu)=2-\frac{\mu_{n-1}}{\mu_n}-\frac{\mu_{n+1}}{\mu_n}$, $n\in \mathbb{N}$, and $\{\mu_n\}$ is a strictly positive sequence of real numbers such that $\mu_0=0$. Unifying both the improved discrete Hardy inequalities (\ref{IMHI2}, \ref{IMHIGN3}), we \cite{DASMAN} have obtained an extended improved discrete Hardy's inequality in one dimension as follows
\begin{align}{\label{GIMHI}}
\displaystyle\sum_{n=1}^{\infty}\frac{|A_{n}-A_{n-1}|^{2}}{\lambda_{n}}\geq \displaystyle\sum_{n=1}^{\infty}w_{n}(\lambda, \mu )|A_n|^{2},
\end{align} where $\lambda=\{\lambda_n\}$ is a real sequence such that $\lambda_{n}>0$, $n\in \mathbb{N}$ and $w_{n}(\lambda, \mu)$ has the following expression
$w_{n}(\lambda, \mu) =\frac{1}{\lambda_{n}}+\frac{1}{\lambda_{n+1}}-\frac{\mu_{n-1}}{\lambda_{n}\mu_{n}}-\frac{\mu_{n+1}}{\lambda_{n+1}\mu_{n}}$.
It is pertinent to mentioned here that our inequality (\ref{GIMHI}) not only gives the above inequalities (\ref{IMHI2}, \ref{IMHIGN3}) but also the improved discrete Hardy inequality with power weights, which is established recently by Gupta \cite{SGA}. He proved that for $\alpha\in \{0\}\cup[1/3, 1)$ and $\beta=\frac{1-\alpha}{2}$
\begin{align}{\label{GUPI}}
\displaystyle\sum_{n=1}^{\infty}|A_{n}-A_{n-1}|^{2}n^{\alpha} &\geq \displaystyle\sum_{n=1}^{\infty}w_{n}(\alpha,\beta)|A_{n}|^{2}>\frac{(\alpha-1)^2}{4}\sum_{n=1}^{\infty}\frac{|A_{n}|^{2}}{n^2}{n}^{\alpha}, ~~~\mbox{holds}
\end{align}where $w_{1}(\alpha,\beta):=1+2^{\alpha}-2^{\alpha+\beta}$, and $w_{n}(\alpha,\beta)=n^{\alpha}\Big[1+\Big(1+\frac{1}{n}\Big)^{\alpha}-\Big(1-\frac{1}{n}\Big)^{\beta}-\Big(1+\frac{1}{n}\Big)^{\alpha+\beta}\Big]$ for $n\geq2$. We note that the inequality (\ref{GUPI}) is indeed a Copson's inequality (\ref{alphaCOPI}) for $q_n=1$, $n\in \mathbb{N}$. That means the Copson's inequality for $q_n=1$, $n\in \mathbb{N}$ (i.e. Hardy's inequality with power weights) achieved an improvement when $\alpha\in \{0\}\cup[1/3, 1)$. This leads to an important question given below, and which will be our first investigation in this paper:
\begin{center}
Q(a) \emph{Is it possible to extend the interval of $\alpha$ for which improvement of (\ref{GUPI}) still possible?}
\end{center}
Many authors studying improvement of discrete Hardy inequality in different contexts such as in graph setting (see \cite{MKRGRAPH}, \cite{MKRRELGRAPH}), in $\mathbb{Z}^d$ (see \cite{HUANGYE}, \cite{MKRGRAPH}) etc. But there is no such study has been found on the improvement of the Copson inequality (\ref{alphaCOPI}). Although, we have considered the improvement of the Copson inequality (\ref{alphaCOPI}) for $q_n=n$, $n\in \mathbb{N}$ in \cite{DASMAN}, and have shown that it admits an improvement only when $\alpha=\frac{1}{2}$ (or $c=\frac{3}{2}$). Again in our another work \cite{DASMAN2}, we were able to achieve an improvement of the Copson inequality (\ref{alphaCOPI}) for $q_n=n^2$, and $q_n=n^3$, $n\in \mathbb{N}$ and when $\alpha=\frac{1}{2}$ only. With these investigations following questions may be asked.
\begin{center}
Q(b) \emph{Can we improve the Copson inequality (\ref{alphaCOPI}) for an arbitrary sequence $\{q_n\}$?}
\end{center}If we get a favourable answer to this question $Q(b)$ then
\begin{center}
Q(c) \emph{What are the possible values of $\alpha$ for which (\ref{alphaCOPI}) gets an improvement?}
\end{center}
Therefore the main objective of this paper is to supply the possible answers of the above proposed queries $(Q(a)$ to $Q(c))$. To reach our objectives first we give an affirmative answer to the question $Q(a)$. That is we prove that the range of $\alpha$ from $\{0\}\cup[1/3, 1)$ can be extended to a larger one $[0, 1)$ for which the improvement of (\ref{GUPI}) is still possible. Consequently, we strengthen the result of Gupta \cite{SGA}. When providing an answer to $Q(b)$, we have seen that for any decreasing sequence $\{q_n\}$, the improvement of (\ref{alphaCOPI}) is attainable for $\alpha\in [1/3, 1)$. But we couldn't manage to answer this question for an arbitrary increasing sequence $\{q_n\}$. Rather, we have shown that there exists improvements of (\ref{alphaCOPI}) for some strictly increasing sequences $\{q_n\}$ such as $q_n=n$, and $q_n=n^3$, $n\in \mathbb{N}$. We prove that the Copson's inequality (\ref{alphaCOPI}) for $q_n=n$ achieves an improvement when $\alpha\in [17/50, 1)$, whereas for $q_n=n^3$, Copson's inequality (\ref{alphaCOPI}) gets an improvement when $\alpha\in [0, 1/2]$.\\
The paper is organized as follows. In Section 2, we establish an improved Hardy's inequality with power weights $n^\alpha$ for $\alpha\in [0, 1)$. Section 3 deals with the improvement of the Copson inequality (\ref{alphaCOPI}) for an arbitrary sequence $\{q_n\}$. In fact, in first subsection we have shown an improved Copson inequality (\ref{alphaCOPI}) for an arbitrary decreasing sequence $\{q_n\}$, and in another subsection, we have proved an improved Copson inequality (\ref{alphaCOPI}) for some increasing sequences $\{q_n\}$.

\section{Improved Hardy's inequality with power weights $n^\alpha$ for $\alpha\in [0, 1)$}
We have seen that the Copson inequality (\ref{alphaCOPI}) in case of $q_n=1$ for $n\in \mathbb{N}$ is equivalent to the inequality (\ref{GUPI}) established by Gupta \cite{SGA}. More precisely, he has obtained an improvement of the inequality (\ref{GUPI}) for $\alpha\in \{0\}\cup[\frac{1}{3}, 1)$. In this section, we strengthen the result of Gupta \cite{SGA} and prove that the improvement of (\ref{GUPI}) is possible for all $\alpha\in [0, 1)$. In fact we have the following result. Let us begin with the following lemma.
\begin{Lemma}{\label{W-1gater}}
Let $\alpha\in[0, 1)$. Then we have
\begin{align*}
1+2^{\alpha}-2^{\frac{1}{2}+\frac{\alpha}{2}}& >\frac{1}{4}(\alpha-1)^2.
\end{align*}
\end{Lemma}
\begin{proof}
Let us consider a function $f(\alpha)$ defined as below:
\begin{align*}
f(\alpha)&=1+2^{\alpha}-2^{\frac{1}{2}+\frac{\alpha}{2}}-\frac{1}{4}(\alpha-1)^2.
\end{align*}
Then we have the following derivative:
\begin{align*}
f'(\alpha)&=\Big(\frac{\log2}{2}(2^{\alpha+1}-2^{\frac{1}{2}+\frac{\alpha}{2}})+\frac{(1-\alpha)}{2}\Big)>0\\
&\Big[\mbox{since},~~2^{\alpha+1}>2^{\frac{1+\alpha}{2}}~~\mbox{for~all}~~\alpha\in[0, 1)\Big]
\end{align*} Hence $f'(\alpha)$ is increasing in $[0, 1)$. Since $f(0)>0$, so we conclude that $f(\alpha)>0$ for $\alpha\in[0, 1)$. Hence the proof.
\end{proof}

\begin{Theorem}Let $\{A_n\}\in C_c(\mathbb{N}_0)$ be such that $A_{0}=0$ and $\alpha\in[0, 1)$. Then we have the following improvement
\begin{align}{\label{GUPI01}}
\displaystyle\sum_{n=1}^{\infty}|A_{n}-A_{n-1}|^{2}n^{\alpha} &\geq \displaystyle\sum_{n=1}^{\infty}w_{n}(\alpha)|A_{n}|^{2}>\frac{(\alpha-1)^2}{4}\sum_{n=1}^{\infty}\frac{|A_{n}|^{2}}{n^2}{n}^{\alpha},
\end{align}where  $w_{1}(\alpha):=1+2^{\alpha}-2^{\frac{1+\alpha}{2}}$, and for $ n\geq2$
\begin{align*}
w_{n}(\alpha)& =n^{\alpha}\Big[1+\Big(1+\frac{1}{n}\Big)^{\alpha}-\Big(1-\frac{1}{n}\Big)^{\frac{1-\alpha}{2}}-\Big(1+\frac{1}{n}\Big)^{\frac{1+\alpha}{2}}\Big].
\end{align*}
\end{Theorem}
\begin{proof}
The LHS of the inequality $(\ref{GUPI01})$ immediately follows from (\ref{GIMHI}) by choosing $\lambda_{n}=\frac{1}{n^{\alpha}}$ and $\mu_{n}=n^{\frac{1-\alpha}{2}}$. Throughout the proof we consider only $\alpha\in(0, 1)$ as for $\alpha=0$ improvement is already achieved in inequality (\ref{IMHI2}). For $n=1$, the improvement is easily followed from Lemma \ref{W-1gater}. In fact we have
\begin{align*}
w_{1}(\alpha)&=1+2^{\alpha}-2^{\frac{1+\alpha}{2}}> \frac{1}{4}(\alpha-1)^2.
\end{align*}
For $n\geq 2$, we define a real valued continuous function $R(x)$ on $x\in[0, \frac{1}{2}]$ as below
\begin{align*}
R(x)=&1+(1+x)^{\alpha}-(1-x)^{\frac{1-\alpha}{2}}-(1+x)^{\frac{1+\alpha}{2}}.
\end{align*}
Then we have
\begin{align*}
R'(x)=&\alpha(1+x)^{\alpha-1}+\frac{1-\alpha}{2}(1-x)^{\frac{-1-\alpha}{2}}-\frac{1+\alpha}{2}(1+x)^{\frac{\alpha-1}{2}},\\
R''(x)=&\alpha(\alpha-1)(1+x)^{\alpha-2}+\big(\frac{1-\alpha^{2}}{4}\big)\Big[(1-x)^{\frac{-3-\alpha}{2}}+(1+x)^{\frac{\alpha-3}{2}}\Big],\\
R'''(x)=&\alpha(1-\alpha)(2-\alpha)(1+x)^{\alpha-3}+\big(\frac{1-\alpha^{2}}{4}\big)\Big[\frac{3+\alpha}{2}(1-x)^{\frac{-5-\alpha}{2}}-
\frac{3-\alpha}{2}(1+x)^{\frac{\alpha-5}{2}}\Big],\\
=&I+\big(\frac{1-\alpha^{2}}{4}\big)f(x),
\end{align*}where
\begin{align*}
I&=\alpha(1-\alpha)(2-\alpha)(1+x)^{\alpha-3},\\
f(x)=&\frac{3+\alpha}{2}(1-x)^{\frac{-5-\alpha}{2}}-\frac{3-\alpha}{2}(1+x)^{\frac{\alpha-5}{2}},~~~\mbox{and hence}\\
f'(x)=&\frac{3+\alpha}{2}\frac{5+\alpha}{2}(1-x)^{\frac{-7-\alpha}{2}}+\frac{3-\alpha}{2}\frac{5-\alpha}{2}(1+x)^{\frac{\alpha-7}{2}}.
\end{align*}
Clearly $f'(x)>0$ for $x\in[0, \frac{1}{2}]$ and $\alpha\in(0, 1)$. This says that $f(x)$ is increasing. Since $f(0)=0$, so we have $f(x)\geq 0$ for $x\in[0, \frac{1}{2}]$.
Also since $I>0$ for $\alpha\in (0, 1)$ therefore we asserts that $R'''(x)>0$ for $x\in[0, \frac{1}{2}]$.\\
Now denote $S(x)=R(x)-\frac{(\alpha-1)^2}{4}x^{2}$. Then for $x\in[0, \frac{1}{2}]$, we get
\begin{align*}
S'(x)=&R'(x)-\frac{(\alpha-1)^2}{2}x,\\
S''(x)=&R''(x)-\frac{(\alpha-1)^2}{2},\\
S'''(x)=&R'''(x).
\end{align*}
Since $R'''(x)>0$ so we have $S'''(x)>0$ in $[0, \frac{1}{2}]$. Note that $S''(0)=S'(0)=S(0)=0$. This tells that $S(x)>0$ for $x\in(0, \frac{1}{2}]$, consequently we have $R(x)>\frac{(\alpha-1)^2}{4}x^{2}$ for $x\in(0, \frac{1}{2}]$. We choose $x=\frac{1}{n}$ for $n\geq 2$, then we get $R(\frac{1}{n})>\frac{(\alpha-1)^2}{4n^{2}}$. Since $w_{n}(\alpha)=n^{\alpha}R(\frac{1}{n})$ then we have for $n\geq 2$ and $\alpha\in(0, 1)$
\begin{align*}
&w_{n}(\alpha)>\frac{(\alpha-1)^2}{4}{n}^{\alpha-2}.
\end{align*}
This completes proof of the theorem.
\end{proof}

\section{Improvement of general Copson's inequality}
In this section, we study the improvement of the Copson's inequality (\ref{alphaCOPI}) in two cases. First we consider the inequality (\ref{alphaCOPI}) for decreasing sequence $\{q_n\}$, and second for some increasing sequences $\{q_n\}$, especially $q_n=n$ and $q_n=n^3$ for $n\in \mathbb{N}$. At this point, we indicate that for arbitrary increasing sequence $\{q_n\}$, we were not able to conclude the improvement of (\ref{alphaCOPI}). Although in our previous works (\cite{DASMAN}, \cite{DASMAN2}), we were only able to prove that the inequality (\ref{alphaCOPI}) admits improvement in a special case of $\alpha=\frac{1}{2}$ $(c=\frac{3}{2})$, and for some increasing sequences such as $q_n=n, n^2, n^3$, $n\in \mathbb{N}$. Let us start with an arbitrary positive decreasing sequence $\{q_n\}$.
\subsection{For decreasing sequence $\{q_n\}$}
Let $\alpha\in[0,1)$, and $\{\mu_{n}\}$ be a strictly positive sequence of real numbers. Then we define a weight sequence $w_{n}(q, Q)$ as below:
\begin{align*}
w_{n}(q,Q)& =\frac{Q_{n}^{\alpha}}{q_{n}}+\frac{Q_{n+1}^{\alpha}}{q_{n+1}}-\frac{\mu_{n-1}Q_{n}^{\alpha}}{q_{n}\mu_{n}}-\frac{\mu_{n+1}Q_{n+1}^{\alpha}}{q_{n+1}\mu_{n}}.
\end{align*}
Now with the choice of $\mu_{n}=Q_{n}^{\frac{1-\alpha}{2}}$, weight sequence $w_{n}(q, Q)$ can be rewritten as below:
\begin{align*}
w_{n}(q, Q)& =\frac{Q_{n}^{\alpha}}{q_{n}}T_{n},
\end{align*}
where of course
\begin{align*}
T_{n}& =1+\frac{q_{n}}{q_{n+1}}\big(\frac{Q_{n+1}}{Q_{n}}\big)^{\alpha}-\big(\frac{Q_{n-1}}{Q_{n}}\big)^{\frac{1-\alpha}{2}}-
\frac{q_{n}}{q_{n+1}}\big(\frac{Q_{n+1}}{Q_{n}}\big)^{\frac{1+\alpha}{2}}\\
&=1+\frac{q_{n}}{q_{n+1}}\Big(1+\frac{q_{n+1}}{Q_{n}}\Big)^{\alpha}-\Big(1-\frac{q_{n}}{Q_{n}}\Big)^{\frac{1-\alpha}{2}}-
\frac{q_{n}}{q_{n+1}}\Big(1+\frac{q_{n+1}}{Q_{n}}\Big)^{\frac{1+\alpha}{2}}.
\end{align*}
Expanding $T_{n}$ as an infinite series, we have
\begin{align*}
T_{n}&=\displaystyle\sum_{k=2}^{\infty}\Big[\binom{\alpha}{k}\frac{q_{n}q_{n+1}^{k-1}}{Q_{n}^{k}}-\binom{\frac{1+\alpha}{2}}{k}
\frac{q_{n}q_{n+1}^{k-1}}{Q_{n}^{k}}-(-1)^{k}\binom{\frac{1-\alpha}{2}}{k}\frac{q_{n}^{k}}{Q_{n}^{k}}\Big]\\
&=\frac{1-\alpha^{2}}{8}\frac{q_{n}^{2}}{Q_{n}^{2}}-\frac{(1-\alpha)(3\alpha-1)}{8}\frac{q_{n}q_{n+1}}{Q_{n}^{2}}+\displaystyle\sum_{k=3}^{\infty}C_{k}(q,\alpha),
\end{align*}
where
\begin{align}{\label{CK-equation}}
C_{k}(q,\alpha)&=\Big[\binom{\alpha}{k}\frac{q_{n}q_{n+1}^{k-1}}{Q_{n}^{k}}-\binom{\frac{1+\alpha}{2}}{k}
\frac{q_{n}q_{n+1}^{k-1}}{Q_{n}^{k}}-(-1)^{k}\binom{\frac{1-\alpha}{2}}{k}\frac{q_{n}^{k}}{Q_{n}^{k}}\Big].
\end{align}
Notice that $(1-\alpha)(3\alpha-1)\geq0$ for $\alpha\in[\frac{1}{3},1)$. Hence if $\alpha\in[\frac{1}{3}, 1)$ then decreasing-ness of $\{q_{n}\}$ implies that
\begin{align}{\label{T equation}}
T_{n}&\geq\frac{(\alpha-1)^{2}}{4}\frac{q_{n}^{2}}{Q_{n}^{2}}+\displaystyle\sum_{k=3}^{\infty}C_{k}(q,\alpha).
\end{align}
Now our aim is to show that $C_{k}(q,\alpha)$ is positive for $k\geq 3$. With this aim, we first prove the following lemma.
\begin{Lemma}{\label{Positive}}
Let $\alpha\in[\frac{1}{3},1)$ and $\{q_{n}\}$ be decreasing sequence of positive real numbers. Then $C_{k}(q,\alpha)$ is strictly positive for $k\geq3$, $k\in\mathbb{N}$.
\end{Lemma}
\begin{proof}
Denote $\gamma_{1}=\frac{1-\alpha}{2}$ and $\gamma_{2}=\frac{1+\alpha}{2}$ for $\alpha\in[\frac{1}{3},1)$. Then one can re-write $C_{k}(q,\alpha)$ as below:
\begin{align*}
C_{k}(q,\alpha)&=(-1)^{k-1}\frac{\alpha}{k!}\frac{q_{n}q_{n+1}^{k-1}}{Q_{n}^{k}}\displaystyle\prod_{i=1}^{k-1}(i-\alpha)
+\frac{\gamma_{1}}{k!}\frac{q_{n}^{k}}{Q_{n}^{k}}\displaystyle\prod_{i=1}^{k-1}(i-\gamma_{1})+(-1)^{k}\frac{\gamma_{2}}{k!}
\frac{q_{n}q_{n+1}^{k-1}}{Q_{n}^{k}}\displaystyle\prod_{i=1}^{k-1}(i-\gamma_{2}).
\end{align*}
We now divide $k\geq 3$, $k\in\mathbb{N}$ in two cases, one is odd natural number $k$, and another is even natural number $k$.\\
\emph{\textbf{Case I: $k$ is odd.}}\\
Note that when $k$ is odd then
\begin{align*}
C_{k}(q,\alpha)&=\frac{\alpha}{k!}\frac{q_{n}q_{n+1}^{k-1}}{Q_{n}^{k}}\displaystyle\prod_{i=1}^{k-1}(i-\alpha)
+\frac{\gamma_{1}}{k!}\frac{q_{n}^{k}}{Q_{n}^{k}}\displaystyle\prod_{i=1}^{k-1}(i-\gamma_{1})-\frac{\gamma_{2}}{k!}
\frac{q_{n}q_{n+1}^{k-1}}{Q_{n}^{k}}\displaystyle\prod_{i=1}^{k-1}(i-\gamma_{2})\\
&=I_{1}+I_{2}-I_{3},
\end{align*}where
$I_{1}=\frac{\alpha}{k!}\frac{q_{n}q_{n+1}^{k-1}}{Q_{n}^{k}}\displaystyle\prod_{i=1}^{k-1}(i-\alpha)$,
$I_{2}=\frac{\gamma_{1}}{k!}\frac{q_{n}^{k}}{Q_{n}^{k}}\displaystyle\prod_{i=1}^{k-1}(i-\gamma_{1})$, and
$I_{3}=\frac{\gamma_{2}}{k!}\frac{q_{n}q_{n+1}^{k-1}}{Q_{n}^{k}}\displaystyle\prod_{i=1}^{k-1}(i-\gamma_{2}).$\\
Now for odd $k\geq3$, and $\alpha\in[\frac{1}{3},1)$, it is immediate that $I_{1}>0$. But the following difference gives
\begin{align*}
I_{2}-I_{3}&=\frac{\gamma_{2}}{k!}\frac{q_{n}^{k}}{Q_{n}^{k}}\displaystyle\prod_{i=1}^{k-1}(i-\gamma_{2})\Big[\frac{\gamma_{1}}{\gamma_{2}}\displaystyle\prod_{i=1}^{k-1}
\big(\frac{i-\gamma_{1}}{i-\gamma_{2}}\big)-\frac{q_{n+1}^{k-1}}{q_{n}^{k-1}}\Big]\\
&=\frac{\gamma_{2}}{k!}\frac{q_{n}^{k}}{Q_{n}^{k}}\displaystyle\prod_{i=1}^{k-1}(i-\gamma_{2})L(q,\alpha),
\end{align*}where we denote $L(q,\alpha)=\frac{\gamma_{1}}{\gamma_{2}}\displaystyle\prod_{i=1}^{k-1}
\big(\frac{i-\gamma_{1}}{i-\gamma_{2}}\big)-\frac{q_{n+1}^{k-1}}{q_{n}^{k-1}}=\displaystyle\prod_{i=2}^{k-1}
\big(\frac{i-\gamma_{1}}{i-\gamma_{2}}\big)-\frac{q_{n+1}^{k-1}}{q_{n}^{k-1}}$. Since $\gamma_2>\gamma_1$, and $\{q_n\}$ is decreasing, so we have
\begin{align*}
&L(q,\alpha)>1-\frac{q_{n+1}^{k-1}}{q_{n}^{k-1}}=\frac{q_{n}^{k-1}-q_{n+1}^{k-1}}{q_{n}^{k-1}}\geq 0,
\end{align*}
which shows that $L(q,\alpha)>0$. Hence $I_{2}-I_{3}>0$. Therefore $C_{k}(q,\alpha)>0$ for all odd $k\geq3$, $k\in\mathbb{N}$.\\
\emph{\textbf{Case II: $k$ is even.}}\\
In this case, using the similar expressions for $I_1$, $I_2$, and $I_3$ in the earlier case, $C_{k}(q,\alpha)$ can be written as follows:
\begin{align*}
C_{k}(q,\alpha)&=-\frac{\alpha}{k!}\frac{q_{n}q_{n+1}^{k-1}}{Q_{n}^{k}}\displaystyle\prod_{i=1}^{k-1}(i-\alpha)+
\frac{\gamma_{1}}{k!}\frac{q_{n}^{k}}{Q_{n}^{k}}\displaystyle\prod_{i=1}^{k-1}(i-\gamma_{1})+
\frac{\gamma_{2}}{k!}\frac{q_{n}q_{n+1}^{k-1}}{Q_{n}^{k}}\displaystyle\prod_{i=1}^{k-1}(i-\gamma_{2})\\
&=-I_{1}+I_{2}+I_{3}\\
&=\frac{\gamma_{1}}{k!}\Big(\displaystyle\prod_{i=1}^{k-1}(i-\alpha)\Big)\Big[\frac{q_{n}^{k}}{Q_{n}^{k}}\displaystyle\prod_{i=1}^{k-1}\big(\frac{i-\gamma_{1}}{i-\alpha}\big)-
\frac{\alpha}{\gamma_{1}}\frac{q_{n}q_{n+1}^{k-1}}{Q_{n}^{k}}\Big]+I_{3}\\
&=\frac{\gamma_{1}}{k!}\Big(\displaystyle\prod_{i=1}^{k-1}(i-\alpha)\Big)M(q,\alpha)+I_{3},
\end{align*}where we denote $M(q,\alpha)=\frac{q_{n}^{k}}{Q_{n}^{k}}\displaystyle\prod_{i=1}^{k-1}\big(\frac{i-\gamma_{1}}{i-\alpha}\big)-\frac{\alpha}{\gamma_{1}}\frac{q_{n}q_{n+1}^{k-1}}{Q_{n}^{k}}$. Since $\frac{i-\gamma_{1}}{i-\alpha}\geq1$ for $\alpha\in[1/3,1)$ and $\{q_{n}\}$ is decreasing, so we have
\begin{align*}
M(q,\alpha)& \geq\frac{q_{n}^{k}M_{1}(\alpha)}{\gamma_{1}Q_{n}^{k}\displaystyle\prod_{i=1}^{7}(i-\alpha)}.
\end{align*}
where we choose $M_{1}(\alpha)$ as below
\begin{align*}
M_{1}(\alpha)&=\Big(\gamma_{1}\displaystyle\prod_{i=1}^{7}(i-\gamma_{1})-\alpha\displaystyle\prod_{i=1}^{7}(i-\alpha)\Big)\\
&=\Big(\frac{1-\alpha}{2}\displaystyle\prod_{i=1}^{7}(i-\frac{1-\alpha}{2})-\alpha\displaystyle\prod_{i=1}^{7}(i-\alpha)\Big)~~~(\because~~ \gamma_{1}=\frac{1-\alpha}{2}).
\end{align*}Simplify the terms of the above expression for $M_{1}(\alpha)$, we get
\begin{align}{\label{M-H-eq}}
M_{1}(\alpha)&=\frac{(1-\alpha)(3\alpha-1)(13-\alpha)}{2^{8}}H(\alpha)
\end{align}
where $H(\alpha)=85\alpha^{5}-1187\alpha^{4}+8654\alpha^{3}-24898\alpha^{2}+46941\alpha-10395$. Hence we have
\begin{align*}
H'(\alpha)&=425\alpha^{4}-4748\alpha^{3}+25962\alpha^{2}-49796\alpha+46941,\\
H''(\alpha)&=1700\alpha^{3}-14244\alpha^{2}+51924\alpha-49796,\\
H'''(\alpha)&=12(425\alpha^{2}-2374\alpha+4327).
\end{align*}
Clearly, $H'''(\alpha)>0$ for $\alpha\in[\frac{1}{3}, 1)$, which implies $H''(\alpha)$ is increasing. Also $H''(1)<0$. Hence $H''(\alpha)<0$ in $\alpha\in[\frac{1}{3}, 1)$ and since $H'(1)>0$, which means that $H'(\alpha)>0$ for any $\alpha\in[\frac{1}{3},1)$. Since $H(\frac{1}{3})>0$ so we conclude that $H(\alpha)>0$ in $[\frac{1}{3}, 1)$. Therefore $M_{1}(\alpha)\geq0$ in $\alpha\in[\frac{1}{3},1)$, which further gives $M(q,\alpha)\geq0$ for all even $k\geq8$ and $\alpha\in[\frac{1}{3},1)$. On the other hand, $I_{3}>0$ and hence  $C_{k}(q,\alpha)>0$ for all even $k\geq8$ and $\alpha\in[\frac{1}{3},1)$. Now for the two remaining even numbers $k=4$, and for $k=6$, $C_{k}(q,\alpha)$ has the following expressions
\begin{align*}
C_{4}(q,\alpha)&=\frac{\gamma_{1}}{4!}\frac{q_{n}^{4}}{Q_{n}^{4}}\displaystyle\prod_{i=1}^{3}(i-\gamma_{1})+
\frac{\gamma_{2}}{4!}\frac{q_{n}q_{n+1}^{3}}{Q_{n}^{4}}\displaystyle\prod_{i=1}^{3}
(i-\gamma_{2})-\frac{\alpha}{4!}\frac{q_{n}q_{n+1}^{3}}{Q_{n}^{4}}\displaystyle\prod_{i=1}^{3}(i-\alpha),\\
C_{6}(q,\alpha)&=\frac{\gamma_{1}}{6!}\frac{q_{n}^{6}}{Q_{n}^{6}}\displaystyle\prod_{i=1}^{5}(i-\gamma_{1})+
\frac{\gamma_{2}}{6!}\frac{q_{n}q_{n+1}^{5}}{Q_{n}^{6}}\displaystyle\prod_{i=1}^{5}
(i-\gamma_{2})-\frac{\alpha}{6!}\frac{q_{n}q_{n+1}^{5}}{Q_{n}^{6}}\displaystyle\prod_{i=1}^{5}(i-\alpha).
\end{align*}
Since $q_{n}\geq q_{n+1}$ for $n\in\mathbb{N}$ so one can write
\begin{align}{\label{c-4}}
C_{4}(q,\alpha)&\geq J_{1}(\alpha)\frac{q_{n}q_{n+1}^{3}}{Q_{n}^{4}},
\end{align} where we denote
\begin{align*}
J_{1}(\alpha)=&\frac{\gamma_{1}}{4!}\displaystyle\prod_{i=1}^{3}(i-\gamma_{1})+\frac{\gamma_{2}}{4!}\displaystyle
\prod_{i=1}^{3}(i-\gamma_{2})-\frac{\alpha}{4!}\displaystyle\prod_{i=1}^{3}(i-\alpha).
\end{align*}Plugging the values of $\gamma_{1}=\frac{1-\alpha}{2}$, $\gamma_{2}=\frac{1+\alpha}{2}$ in $J_{1}(\alpha)$ we get a simplified form as below
\begin{align*}
J_{1}(\alpha)=\frac{1}{4!8}(\alpha^2-6\alpha+5)(7\alpha^2-6\alpha+3).
\end{align*} It is clear that $J_1(\alpha)>0$ in $[\frac{1}{3}, 1)$ which shows that $C_{4}(q,\alpha)>0$ in $[\frac{1}{3}, 1)$.\\
Similarly one has
\begin{align}
C_{6}(q,\alpha)&\geq J_{2}(\alpha)\frac{q_{n}q_{n+1}^{5}}{Q_{n}^{6}},
\end{align}
where we denote
\begin{align*}
J_{2}(\alpha)=&\frac{\gamma_{1}}{6!}\displaystyle\prod_{i=1}^{5}(i-\gamma_{1})+\frac{\gamma_{2}}{6!}\displaystyle\prod_{i=1}^{5}
(i-\gamma_{2})-\frac{\alpha}{6!}\displaystyle\prod_{i=1}^{5}(i-\alpha).
\end{align*}Inserting the values of $\gamma_{1}$ and $\gamma_{2}$ in $J_{2}(\alpha)$, we obtain
\begin{align*}
J_{2}(\alpha)=&\frac{1}{6!32}(31\alpha^{6}-480\alpha^{5}+2515\alpha^{4}-7200\alpha^{3}+8029\alpha^{2}-3840\alpha+945).
\end{align*}Differentiating $J_{2}(\alpha)$ up to four times, we get
\begin{align*}
J'_{2}(\alpha)&=\frac{1}{6!32}(186\alpha^{5}-2400\alpha^{4}+10060\alpha^{3}-21600\alpha^{2}+16058\alpha-3840),\\
J''_{2}(\alpha)&=\frac{1}{6!32}(930\alpha^{4}-9600\alpha^{3}+30180\alpha^{2}-43200\alpha+16058),\\
J'''_{2}(\alpha)&=\frac{1}{6!32}(3720\alpha^{3}-28800\alpha^{2}+60360\alpha-43200),\\
J''''_{2}(\alpha)&=\frac{1}{192}(93\alpha^{2}-480\alpha+503).
\end{align*}\\
Note that $J''''_{2}(\alpha)>0$ in $[\frac{1}{3}, 1)$ which implies that $J'''_{2}(\alpha)$ is increasing in $[\frac{1}{3}, 1)$ but $J'''_{2}(1)<0$, which says that $J''_{2}(\alpha)$ is decreasing in $[\frac{1}{3}, 1)$ but $J''_{2}(\frac{27}{50})>0$ implies that $J''_{2}(\alpha)>0$ in $[\frac{1}{3}, \frac{27}{50}]$. That is $J'_{2}(\alpha)$ is increasing in $[\frac{1}{3}, \frac{27}{50}]$ but $J'_{2}(\frac{1}{3})>0$ and $J_{2}(\frac{1}{3})>0$, which finally gives $J_{2}(\alpha)>0$ in $[\frac{1}{3}, \frac{27}{50}]$. One can also write $J_{2}(\alpha)$ as $J_{2}(\alpha)=\frac{(9-\alpha)f(\alpha)}{6!32}$,
where $f(\alpha)=-31\alpha^{5}+201\alpha^{4}-706\alpha^{3}+846\alpha^{2}-415\alpha+105$. Then
\begin{align*}
 f'(\alpha)&=-155\alpha^{4}+804\alpha^{3}-2118\alpha^{2}+1692\alpha-415<-6(25\alpha^{4}-134\alpha^{3}+353\alpha^{2}-282\alpha+69).
 \end{align*}We observe that $25\alpha^{4}-134\alpha^{3}+353\alpha^{2}-282\alpha+69>0$ in $\alpha\in[27/50, 1)$. Therefore $f'(\alpha)<0$. Also since $f(1)=0$, so
 $f(\alpha)>0$ in $\alpha\in[27/50, 1)$. Consequently $J_{2}(\alpha)>0$ $\forall\alpha\in[27/50,1)$, and hence $J_{2}(\alpha)>0$ for $\alpha\in[\frac{1}{3}, 1)$. This gives $C_{6}(q,\alpha)>0$ for $\alpha\in[\frac{1}{3}, 1)$. Hence we prove the Lemma.
\end{proof}
Now we have another result.
\begin{Lemma}{\label{STRICTLYP}}
If $\{q_{n}\}$ is decreasing and $\alpha\in[\frac{1}{3},1)$, then $w_{n}(q, Q)>\frac{(\alpha-1)^{2}}{4}\frac{q_{n}}{Q_{n}^{2-\alpha}}.$
\end{Lemma}
\begin{proof}
Using the Lemma \ref{Positive}, inequality (\ref{T equation}) implies that $T_{n}>\frac{(\alpha-1)^{2}}{4}\frac{q_{n}^{2}}{Q_{n}^{2}}$ for $n\geq 1$, $n\in \mathbb{N}$.\\
Since $w_{n}(q, Q)=\frac{Q_{n}^{\alpha}}{q_{n}}T_{n}$, which proves that
\begin{align*}
w_{n}(q, Q)& >\frac{(\alpha-1)^{2}}{4}\frac{q_{n}}{Q_{n}^{2-\alpha}}.
\end{align*}
This proves the lemma.
\end{proof}
We are now in a position to present main result in this subsection. In fact we have the following theorem.
\begin{Theorem}{\label{COPTHEO}}
Let $A_{n}\in C_{c}(\mathbb{N}_0)$ be such that $A_{0}=0$ and $\alpha\in[\frac{1}{3},1)$. Then for any decreasing sequence $\{q_{n}\}$ we have the following improved inequality
\begin{align}{\label{COPGENIMP}}
&\displaystyle\sum_{n=1}^{\infty}\frac{Q_{n}^{\alpha}|A_n-A_{n-1}|^{2}}{q_{n}}\geq\displaystyle\sum_{n=1}^{\infty}w_{n}(q,Q)|A_{n}|^{2}
>\frac{(\alpha-1)^{2}}{4}\displaystyle\sum_{n=1}^{\infty}\frac{q_{n}}{Q_{n}^{2-\alpha}}|A_{n}|^{2},
\end{align}
where
\begin{align*}
w_{n}(q,Q)& =\frac{Q_{n}^{\alpha}}{q_{n}}\Big[1+\frac{q_{n}}{q_{n+1}}\Big(1+\frac{q_{n+1}}{Q_{n}}\Big)^{\alpha}-\Big(1-\frac{q_{n}}{Q_{n}}\Big)^{\frac{1-\alpha}{2}}
-\frac{q_{n}}{q_{n+1}}\Big(1+\frac{q_{n+1}}{Q_{n}}\Big)^{\frac{1+\alpha}{2}}\Big].
\end{align*}
\end{Theorem}
\begin{proof}
The LHS of (\ref{COPGENIMP}) is immediately followed by choosing $\lambda_{n}=\frac{q_{n}}{Q_{n}^{\alpha}}$ and $\mu_{n}=Q_{n}^{\frac{1-\alpha}{2}}$ in inequality (\ref{GIMHI}). The RHS of inequality (\ref{COPGENIMP}) is achieved from Lemma \ref{STRICTLYP}. This completes proof of the theorem.
\end{proof}

\subsection{For some increasing sequences $\{q_n\}$}
In this segment, we discuss the improvement of the Copson inequality (\ref{alphaCOPI}) for some increasing sequences such as $q_n=n$, and $q_n=n^3$ for $n\in \mathbb{N}$. We are not able to address the improvement study for an arbitrary increasing sequence $\{q_n\}$ in this paper. But there is a possibility of improvement of Copson inequality for increasing sequence. Indeed, we prove that Copson inequality (\ref{alphaCOPI}) achieve an improvement for $q_n=n$, and $q_n=n^3$ for $n\in \mathbb{N}$.
\subsubsection{For increasing sequence $\{n\}$:}
Let us choose $q_n=n$, $n\in \mathbb{N}$. Then we prove the following result, statement of which is given as follows.
\begin{Theorem}{\label{COP-n}}
Let $A_{n}\in C_{c}(\mathbb{N}_0)$ be such that $A_{0}=0$, then for any $\alpha\in[0,1)$ we have the following improved inequality
\begin{align}{\label{COP-ninq}}
&\displaystyle\sum_{n=1}^{\infty}\frac{S_{n}^{\alpha}|A_n-A_{n-1}|^{2}}{n}\geq\displaystyle\sum_{n=1}^{\infty}w_{n}(n,S)|A_{n}|^{2}>\frac{(\alpha-1)^{2}}{4}
\displaystyle\sum_{n=1}^{\infty}\frac{n}{S_{n}^{2-\alpha}}|A_{n}|^{2}
\end{align}
where the weight sequence $w_{n}(n, S)$ is defined as below
\begin{align*}
w_{n}(n, S)& =\frac{S_{n}^{\alpha}}{n}\Big[1+\frac{n}{n+1}\Big(\frac{n+2}{n}\Big)^{\alpha}-\Big(\frac{n-1}{n+1}\Big)^{\frac{1-\alpha}{2}}-\frac{n}{n+1}
\Big(\frac{n+2}{n}\Big)^{\frac{1+\alpha}{2}}\Big].
\end{align*}
\end{Theorem}
To establish this theorem, we need some results given in the form of lemma. Let us begin with the following lemma.
\begin{Lemma}{\label{17/50}}
If $\alpha\in[\frac{17}{50}, 1)$ then we have
\begin{align}{\label{17/50inq}}
&8\alpha(1-\alpha)(2-\alpha)(1+2x)^{\alpha-3}>(3-\alpha)(1-\alpha^{2})(1+2x)^{\frac{\alpha-5}{2}}~~~~\forall x\in[0,\frac{1}{2}].
\end{align}
\end{Lemma}
\begin{proof}
Suppose that $\alpha\in[\frac{17}{50}, 1)$ and $x\in[0,\frac{1}{2}]$. To prove above inequality (\ref{17/50inq}) it is equivalent to show that
\begin{align*}
 &\log\Big(8\alpha(1-\alpha)(2-\alpha)(1+2x)^{\alpha-3}\Big)>\log\Big((3-\alpha)(1-\alpha^{2})(1+2x)^{\frac{\alpha-5}{2}}\Big).
 \end{align*}
Let us denote $P(x,\alpha)$ by the following difference.
\begin{align*}
 P(x,\alpha)&=\log\Big(8\alpha(1-\alpha)(2-\alpha)\Big)-(3-\alpha)\log(1+2x)-\log\Big((3-\alpha)(1-\alpha^{2})\Big)\\
 & \hspace{0.5cm}+\frac{5-\alpha}{2}\log(1+2x)\\
&=\log\Big(\frac{8\alpha(2-\alpha)}{(3-\alpha)(1+\alpha)}\Big)-\frac{1-\alpha}{2}\log(1+2x).
 \end{align*}
Differentiating $P(x,\alpha)$ with respect to $x$, we get
\begin{align*}
\frac{\partial{P(x,\alpha)}}{\partial{x}}=&-\frac{1-\alpha}{1+2x}<0,
\end{align*}
which shows that $P(x,\alpha)$ is decreasing in $x$ for any $\alpha\in (0, 1)$. Also note that $P(1/2, \alpha)=\log\Big(\frac{8\alpha(2-\alpha)}{2^{\frac{1-\alpha}{2}}(3-\alpha)(1+\alpha)}\Big)$. We further denote $P(1/2, \alpha)=G(\alpha)$, where
\begin{align}
G(\alpha)=&\log\Big(\frac{8\alpha(2-\alpha)}{2^{\frac{1-\alpha}{2}}(3-\alpha)(1+\alpha)}\Big) \nonumber\\
=&\log(8\alpha)+\log(2-\alpha)-\log(3-\alpha)-\log(1+\alpha)-\frac{1-\alpha}{2}\log2.
\end{align}
Differentiating $G(\alpha)$ with respect to $\alpha$, then after simplification one obtains
\begin{align*}
G'(\alpha)=\frac{6(1-\alpha)}{\alpha(\alpha+1)(2-\alpha)(3-\alpha)}+\frac{\log2}{2}>0~~~~~~\hspace{0.5cm}\forall\alpha\in(0,1).
\end{align*}
This shows that $G(\alpha)$ is increasing. Again since $G(17/50)>0$ which means that $G(\alpha)>0$ for $\alpha\in[\frac{17}{50}, 1)$ which implies $P(1/2,\alpha)>0$ within the specified domain. Therefore $P(x, \alpha)>0$ for $x\in[0,\frac{1}{2}]$ and $\alpha\in[\frac{17}{50}, 1)$. Hence the lemma is proved.
\end{proof}
In the next result, we obtain a lower bound of $w_{n}(n, S)$. In fact we have the following result.
\begin{Lemma}{\label{W-grater}}
Suppose that $\alpha\in[17/50, 1)$ and $n\in\mathbb{N}$. Then
\begin{align*}
&w_{n}(n,S)>\frac{(\alpha-1)^{2}}{4}\frac{n}{S_{n}^{2-\alpha}}~~~~\mbox{holds}.
\end{align*}
\end{Lemma}
\begin{proof}
First we observe that $w_{n}(n,S)=\frac{S_{n}^{\alpha}}{n}P(n)$, where $P(n)$ has the following form
\begin{align*}
P(n)&=1+\frac{1}{(1+\frac{1}{n})}\Big(1+\frac{2}{n}\Big)^{\alpha}-\Big(\frac{1-\frac{1}{n}}{1+\frac{1}{n}}\Big)^{\frac{1-\alpha}{2}}
-\frac{1}{(1+\frac{1}{n})}\Big(1+\frac{2}{n}\Big)^{\frac{1+\alpha}{2}}.
\end{align*}Now it is enough to show that for each $n\in\mathbb{N}$ the following is holds:
\begin{align}{\label{p(n)geq}}
&P(n)>\frac{(\alpha-1)^{2}}{4}\frac{n^{2}}{S_{n}^{2}}=\frac{(\alpha-1)^{2}}{(n+1)^{2}}.
\end{align}
For $n=1$ in inequality (\ref{p(n)geq}), we denote the difference as $T(\alpha)$, where
\begin{align}
T(\alpha)&=1+\frac{3^{\alpha}-3^{\frac{\alpha+1}{2}}}{2}-\frac{(\alpha-1)^{2}}{4},
\end{align}
which gives $T'(\alpha)=\frac{1-\alpha}{2}+3^{\alpha}\frac{\log3}{4}(2-3^{\frac{1-\alpha}{2}})>0$ for $\alpha\in[17/50, 1)$. Again we see that $T(17/50)>0$ which implies that $T(\alpha)>0$ for any $\alpha\in[17/50,1)$.\\
For $n\geq2$, let us consider a real valued function on $[0,1/2]$ defined as below
\begin{align*}
M(x)&=1+\frac{1}{(1+x)}\Big(1+2x\Big)^{\alpha}-\Big(\frac{1-x}{1+x}\Big)^{\frac{1-\alpha}{2}}-\frac{1}{(1+x)}\Big(1+2x\Big)^{\frac{1+\alpha}{2}}-\frac{(\alpha-1)^{2}x^{2}}{(x+1)^{2}}
\end{align*}
After some simplification, one can write $M(x)$ as follows
\begin{align*}
M(x)&=\frac{N(x)}{x+1},
\end{align*}where we denote
\begin{align*}
N(x)&=x+1+(1+2x)^{\alpha}-(1-x)^{\frac{1-\alpha}{2}}(1+x)^{\frac{1+\alpha}{2}}-(1+2x)^{\frac{1+\alpha}{2}}-(1-\alpha)^{2}x^{2}(1+x)^{-1}.
\end{align*}
Differentiating $N(x)$ consecutively for three times with respect to $x$, we get
\begin{align*}
N'(x)=&1+2\alpha(1+2x)^{\alpha-1}+\frac{1-\alpha}{2}(1-x)^{\frac{-1-\alpha}{2}}(1+x)^{\frac{1+\alpha}{2}}-\frac{1+\alpha}{2}(1-x)^{\frac{1-\alpha}{2}}(1+x)^{\frac{\alpha-1}{2}}\\
&-(1+\alpha)(1+2x)^{\frac{\alpha-1}{2}}+(1-\alpha)^{2}x^{2}(1+x)^{-2}-(1-\alpha)^{2}2x(1+x)^{-1},\\
N''(x)=&\frac{1-\alpha^{2}}{4}(1-x)^{\frac{-3-\alpha}{2}}(1+x)^{\frac{1+\alpha}{2}}+\frac{1-\alpha^{2}}{2}(1-x)^{\frac{-1-\alpha}{2}}(1+x)^{\frac{\alpha-1}{2}}+
4x(1-\alpha)^{2}(x+1)^{-2}\\
&+\frac{1-\alpha^{2}}{4}(1-x)^{\frac{1-\alpha}{2}}(1+x)^{\frac{\alpha-3}{2}}+(1-\alpha^{2})(1+2x)^{\frac{\alpha-3}{2}}-4\alpha(1-\alpha)(1+2x)^{\alpha-2}\\
&-2(\alpha-1)^{2}(x+1)^{-1}-2(1-\alpha)^{2}x^{2}(1+x)^{-3}.
\end{align*}
Further, we have
\begin{align*}
N'''(x)=&\frac{1-\alpha^{2}}{4}\frac{3+\alpha}{2}(1-x)^{\frac{-5-\alpha}{2}}(1+x)^{\frac{1+\alpha}{2}}-\frac{1-\alpha^{2}}{4}
\frac{3-\alpha}{2}(1-x)^{\frac{1-\alpha}{2}}(1+x)^{\frac{\alpha-5}{2}}\\
&+\frac{3(1-\alpha^{2})}{4}\frac{1+\alpha}{2}(1-x)^{\frac{-3-\alpha}{2}}(1+x)^{\frac{\alpha-1}{2}}-\frac{3(1-\alpha^{2})}{4}
\frac{1-\alpha}{2}(1-x)^{\frac{-1-\alpha}{2}}(1+x)^{\frac{\alpha-3}{2}}\\
&+6(1-\alpha)^{2}(1+x)^{-2}-12x(1-\alpha)^{2}(1+x)^{-3}+8\alpha(1-\alpha)(2-\alpha)(1+2x)^{\alpha-3}\\
&+6x^{2}(1-\alpha)^{2}(1+x)^{-4}-(3-\alpha)(1-\alpha^{2})(1+2x)^{\frac{\alpha-5}{2}}.
\end{align*}
Simplifying each terms, one gets
\begin{align*}
N'''(x)&=(1-\alpha^{2})(1-x)^{\frac{-\alpha-5}{2}}(1+x)^{\frac{\alpha-5}{2}}(\alpha+3x)+6(1-\alpha)^{2}(1+x)^{-4}\\
&+8\alpha(1-\alpha)(2-\alpha)(1+2x)^{\alpha-3}-(3-\alpha)(1-\alpha^{2})(1+2x)^{\frac{\alpha-5}{2}}.
\end{align*}
Using Lemma \ref{17/50}, we conclude that $N'''(x)>0$ for $x\in[0, \frac{1}{2}]$ and $\alpha\in[\frac{17}{50}, 1)$. This proves that $N''(x)$ is increasing on $[0, \frac{1}{2}]$. Again since $N''(0)=N'(0)=N(0)=0$, so we have $N(x)>0$ on $x\in(0, \frac{1}{2}]$, and consequently we have $M(x)>0$ for $x\in(0, \frac{1}{2}]$. Hence we prove the lemma.
\end{proof}

\begin{proof}
\emph{\textbf{of Theorem \ref{COP-n}:}}\\
The LHS of inequality (\ref{COP-ninq}) is immediately followed from inequality (\ref{GIMHI}) by choosing $\lambda_{n}=\frac{n}{S_{n}^{\alpha}}$ and $\mu_{n}=S_{n}^{\frac{1-\alpha}{2}}$, where $S_{n}=\frac{n(n+1)}{2}$. The RHS of inequality (\ref{COP-ninq}) is a consequence of Lemma \ref{W-grater}. This completes the proof.
\end{proof}
\subsubsection{For increasing sequence $\{n^3\}$:}
Let us now choose $q_n=n^3$ for $n\in \mathbb{N}$. Then the Copson inequality (\ref{alphaCOPI}) reduces to the following inequality.
\begin{align}{\label{COP-n3inq}}
&\displaystyle\sum_{n=1}^{\infty}\frac{\hat{S}_{n}^{\alpha}|A_n-A_{n-1}|^{2}}{n^{3}}\geq
\frac{(\alpha-1)^{2}}{4}\displaystyle\sum_{n=1}^{\infty}\frac{n^{3}}{\hat{S}_{n}^{2-\alpha}}|A_{n}|^{2},
\end{align}where the associated constant term is best possible. The following theorem shows that the inequality (\ref{COP-n3inq}) admits an improvement.
\begin{Theorem}{\label{COP-n3}}
Let $A_{n}\in C_{c}(\mathbb{N}_0)$ be such that $A_{0}=0$ and $\alpha\in[0,\frac{1}{2}]$. Then we have
\begin{align}{\label{COP-n3inqimp}}
&\displaystyle\sum_{n=1}^{\infty}\frac{\hat{S}_{n}^{\alpha}|A_n-A_{n-1}|^{2}}{n^{3}}\geq
\displaystyle\sum_{n=1}^{\infty}\hat{w}_{n}(n^3, S)|A_{n}|^{2}>\frac{(\alpha-1)^{2}}{4}\displaystyle\sum_{n=1}^{\infty}\frac{n^{3}}{\hat{S}_{n}^{2-\alpha}}|A_{n}|^{2}
\end{align}
where the improved weight sequence $\hat{w}_{n}(n^3, S)$ is given by
\begin{align*}
\hat{w}_{n}(n^3, S)& =\frac{\hat{S}_{n}^{\alpha}}{n^{3}}\Big[1+\big(\frac{n}{n+1}\big)^{3}\Big(\frac{n+2}{n}\Big)^{2\alpha}-
\Big(\frac{n-1}{n+1}\Big)^{1-\alpha}-\big(\frac{n}{n+1}\big)^{3}\Big(\frac{n+2}{n}\Big)^{1+\alpha}\Big].
\end{align*}
\end{Theorem}

\begin{proof}
We can easily establish the LHS of inequality (\ref{COP-n3inqimp}) from inequality (\ref{GIMHI}) by choosing $\lambda_{n}=\frac{n^3}{\hat{S}_{n}^{\alpha}}$, and  $\mu_{n}=\hat{S}_{n}^{\frac{1-\alpha}{2}}$, where $\hat{S}_{n}=\big(\frac{n(n+1)}{2}\big)^{2}$. The RHS of inequality (\ref{COP-n3inqimp}) is a consequence of Lemma \ref{W-gater for n^3}, which is proved below. This completes proof of the theorem.
\end{proof}
\begin{Lemma}{\label{W-gater for n^3}}
If $\alpha\in[0,\frac{1}{2}]$ then for each $n\in\mathbb{N}$, we have
\begin{align*}
&\hat{w}_{n}(n^3, S)>\frac{(\alpha-1)^{2}}{4}\frac{n^{3}}{\hat{S}_{n}^{2-\alpha}}.
\end{align*}
\end{Lemma}
\begin{proof}
Note that one can write $\hat{w}_{n}(n^3, S)=\frac{\hat{S}_{n}^{\alpha}}{n^3}V(n)$, where
\begin{align*}
V(n)&=1+\frac{1}{(1+\frac{1}{n})^{3}}\Big(1+\frac{2}{n}\Big)^{\alpha}-\Big(\frac{1-\frac{1}{n}}{1+\frac{1}{n}}\Big)^{1-\alpha}-
\frac{1}{(1+\frac{1}{n})^{3}}\Big(1+\frac{2}{n}\Big)^{1+\alpha}.
\end{align*}To establish the desired inequality it is enough to prove that for each $n\in\mathbb{N}$
\begin{align}{\label{v(n)geq}}
&V(n)>\frac{(\alpha-1)^{2}}{4}\frac{n^{6}}{\hat{S}_{n}^{2}}=\frac{4(\alpha-1)^{2}n^{2}}{(n+1)^{4}}~~~~~\mbox{holds}.
\end{align}
In case of $n=1$, we denote the difference in inequality (\ref{v(n)geq}) by $T(\alpha)$, where
\begin{align}
T(\alpha)&=1+\frac{3^{2\alpha}-3^{\alpha+1}}{8}-\frac{(\alpha-1)^{2}}{4}.
\end{align}
Now differentiation of $T(\alpha)$ gives
\begin{align*}
T'(\alpha)&=\frac{1-\alpha}{2}+\frac{(3^{2\alpha}2-3^{\alpha+1})\log3}{8},\\
&\geq\frac{1-\alpha}{2}-\frac{3^{\alpha}}{8}\log3~~\Big[\mbox{since},~~3^{2\alpha}\geq3^{\alpha}\Big]\\
&\geq\Big(\frac{1}{4}-\frac{\sqrt3}{8}\log3\Big)>0.
\end{align*}
 That is we see  that $T'(\alpha)>0$ for any $\alpha\in[0, \frac{1}{2}]$. Again we have $T(0)>0$, therefore $T(\alpha)>0$ for any $\alpha\in[0,\frac{1}{2}]$. This proves the inequality (\ref{v(n)geq}).\\
Let us now choose $n\geq2$ for $n\in\mathbb{N}$. We define a real-valued function $V(x)$ on $[0, 1/2]$ as below
\begin{align*}
V(x)&=1+\frac{1}{(1+x)^{3}}\Big(1+2x\Big)^{2\alpha}-\Big(\frac{1-x}{1+x}\Big)^{1-\alpha}-\frac{1}{(1+x)^{3}}\Big(1+2x\Big)^{1+\alpha}-\frac{4(\alpha-1)^{2}x^{2}}{(x+1)^{4}}\\
&=\frac{U(x)}{(x+1)^{4}}, ~~~\mbox{where}
\end{align*}
$U(x)=(x+1)^{4}+(1+x)(1+2x)^{2\alpha}-(1-x)^{1-\alpha}(1+x)^{3+\alpha}-(x+1)(1+2x)^{1+\alpha}-4(1-\alpha)^{2}x^{2}$. In case of $\alpha=0$, we have $U(x)=(x+1)^{4}+(1+x)-(1-x)(1+x)^{3}-(x+1)(1+2x)-4x^{2}=2x^{4}+6x^{3}>0$. Suppose now that $\alpha\in(0, \frac{1}{2}]$. Then repeated differentiation of $U(x)$ gives
\begin{align*}
&U'(x)=4(x+1)^{3}+4\alpha(x+1)(1+2x)^{2\alpha-1}+(1+2x)^{2\alpha}+(1-\alpha)(1+x)^{3+\alpha}(1-x)^{-\alpha}\\
&-(3+\alpha)(1+x)^{2+\alpha}(1-x)^{1-\alpha}-(1+2x)^{1+\alpha}-2(x+1)(1+\alpha)(1+2x)^{\alpha}-8x(\alpha-1)^{2},
\end{align*}
\begin{align*}
&U''(x)=12(x+1)^{2}+\alpha(1-\alpha)(1+x)^{\alpha+3}(1-x)^{-\alpha-1}+2(1-\alpha)(\alpha+3)(x+1)^{\alpha+2}(1-x)^{-\alpha}\\
&+8\alpha(2\alpha-1)(x+1)(2x+1)^{2\alpha-2}+8\alpha(1+2x)^{2\alpha-1}-(\alpha+3)(\alpha+2)(x+1)^{\alpha+1}(1-x)^{1-\alpha}\\
&-4\alpha(\alpha+1)(x+1)(2x+1)^{\alpha-1}-4(\alpha+1)(2x+1)^{\alpha}-8(\alpha-1)^{2},
\end{align*}
and $U'''(x)=R_{1}+R_{2}+R_{3}+R_{4}$, where
\begin{align*}
R_{1}&=24(x+1)-12\alpha(\alpha+1)(1+2x)^{\alpha-1},\\
R_{2}&=3(1-\alpha)(\alpha+2)(\alpha+3)(1+x)^{\alpha+1}(1-x)^{-\alpha}-(\alpha+1)(\alpha+2)(\alpha+3)(1+x)^{\alpha}(1-x)^{1-\alpha},\\
R_{3}&=32\alpha(1-2\alpha)(1-\alpha)(1+x)(1+2x)^{2\alpha-3}-8\alpha(1-2\alpha)(1+2x)^{2\alpha-2}, ~~~\mbox{and}\\
R_{4}&=\alpha(1-\alpha^{2})(1-x)^{-\alpha-2}(1+x)^{\alpha+3}+3\alpha(1-\alpha)(\alpha+3)(1+x)^{2+\alpha}(1-x)^{-1-\alpha}\\
&+8\alpha(1-\alpha^{2})(1+x)(2x+1)^{\alpha-2}-16\alpha(1-2\alpha)(1+2x)^{2\alpha-2}.
\end{align*}
One can immediately sees that $R_{1}>0$ for any $\alpha\in (0, \frac{1}{2}]$. Note that $R_2$ can be simplified as below
\begin{align*}
R_{2}=&(\alpha+1)(\alpha+2)(\alpha+3)(1+x)^{\alpha}(1-x)^{1-\alpha}\Big(\frac{3(1-\alpha)}{1+\alpha}\frac{(1+x)}{(1-x)}-1\Big).
\end{align*}
Since $\frac{3(1-\alpha)}{1+\alpha}\geq1$ for $\alpha\in (0, \frac{1}{2}]$ so we have $R_{2}\geq0$ for any $\alpha\in (0, \frac{1}{2}]$. Again for $\alpha\in (0, \frac{1}{2}]$, we have $R_{3}=8\alpha(1-2\alpha)(1+2x)^{2\alpha-3}\Big((3-4\alpha)+2x(1-2\alpha)\Big)\geq 0$. Since $4(2x+1)^{\alpha-2}>1$, $(1-x)^{-\alpha-2}(1+x)^{\alpha+3}\geq1$ and $(1+x)^{2+\alpha}(1-x)^{-1-\alpha}\geq1$ in $x\in[0, \frac{1}{2}]$, so we have
\begin{align*}
R_{4}&>\frac{1}{4}\Big(4\alpha(1-\alpha^{2})+12\alpha(1-\alpha)(\alpha+3)+8\alpha(1-\alpha^{2})-64\alpha(1-2\alpha)\Big)=\frac{S(\alpha)}{4},
\end{align*}where we denote $S(\alpha)$ by the following function
\begin{align*}
 S(\alpha)&=-24\alpha^{3}+104\alpha^{2}-16\alpha=24\alpha\Big(\alpha-\frac{13-\sqrt{145}}{6}\Big)\Big(\frac{13+\sqrt{145}}{6}-\alpha\Big).
\end{align*}
Clearly $S(\alpha)>0$ for any $\alpha\in(\frac{13-\sqrt145}{6}, \frac{1}{2}]$. Hence $R_{4}>0$ for any $\alpha\in(\frac{13-\sqrt145}{6},\frac{1}{2}]$. Again we note that $R_{4}$ is dominated by the function $R(x)$, where
\begin{align*}
R(x)&=\alpha(1-\alpha^{2})+3\alpha(1-\alpha)(\alpha+3)+8\alpha(1-\alpha^{2})(2x+1)^{\alpha-2}-16\alpha(1-2\alpha)(1+2x)^{2\alpha-2}.
\end{align*} Differentiating $R(x)$, we get
\begin{align*}
R'(x)&=64\alpha(1-\alpha)(1-2\alpha)(1+2x)^{2\alpha-3}-16\alpha(1-\alpha^{2})(2-\alpha)(2x+1)^{\alpha-3}\\
&=16\alpha(1-\alpha^{2})(2-\alpha)(2x+1)^{\alpha-3}\Big(\frac{4(1-2\alpha)}{(1+\alpha)(2-\alpha)}(1+2x)^{\alpha}-1\Big).
\end{align*} Since $4(1-2\alpha)-(1+\alpha)(2-\alpha)=\alpha^{2}-9\alpha+2=\Big(\alpha-\frac{(9-\sqrt{73})}{2}\Big)\Big(\alpha-\frac{(9+\sqrt{73})}{2}\Big)$ so one gets $\frac{4(1-2\alpha)}{(1+\alpha)(2-\alpha)}\geq1$ for any $\alpha\in(0, \frac{(9-\sqrt{73})}{2}]$, and hence $R'(x)\geq0$ for $x\in[0, \frac{1}{2}]$ and $\alpha\in(0, \frac{(9-\sqrt{73})}{2}]$. As $R(0)>0$ so $R(x)\geq0$ for $x\in[0, \frac{1}{2}]$ and $\alpha\in(0, \frac{(9-\sqrt{73})}{2}]$. Hence $R_{4}\geq 0$ for any $\alpha\in(0, \frac{(9-\sqrt{73})}{2}]$. Finally we get $R_{4}\geq 0$ for any $\alpha\in(0, \frac{1}{2}]$ and $x\in[0, \frac{1}{2}]$. This proves that $U'''(x)>0$ $\forall$ $\alpha\in(0,\frac{1}{2}]$.
Therefore $U''(x)$ is increasing on $(0, \frac{1}{2}]$. Also since $U''(0)=U'(0)=U(0)=0$, therefore we get $U(x)>0$ in $\alpha\in(0,\frac{1}{2}]$, and consequently we have $V(x)>0$ for $x\in(0, \frac{1}{2}]$ and $\alpha\in(0,\frac{1}{2}]$. Therefore $V(x)>0$ for $x\in(0, \frac{1}{2}]$ and $\alpha\in[0, \frac{1}{2}]$. This proves the Lemma.
\end{proof}

\end{document}